\documentclass[12pt]{amsart}
\usepackage{amsmath, amscd}
\usepackage{amsfonts}

\begin{document}
\numberwithin{equation}{section}

\def\1#1{\overline{#1}}
\def\2#1{\widetilde{#1}}
\def\3#1{\widehat{#1}}
\def\4#1{\mathbb{#1}}
\def\5#1{\frak{#1}}
\def\6#1{{\mathcal{#1}}}

\newcommand{\w}{\omega}
\newcommand{\Lie}[1]{\ensuremath{\mathfrak{#1}}}
\newcommand{\LieL}{\Lie{l}}
\newcommand{\LieH}{\Lie{h}}
\newcommand{\LieG}{\Lie{g}}
\newcommand{\de}{\partial}
\newcommand{\R}{\mathbb R}
\newcommand{\FH}{{\sf Fix}(H_p)}
\newcommand{\al}{\alpha}
\newcommand{\tr}{\widetilde{\rho}}
\newcommand{\tz}{\widetilde{\zeta}}
\newcommand{\tk}{\widetilde{C}}
\newcommand{\tv}{\widetilde{\varphi}}
\newcommand{\hv}{\hat{\varphi}}
\newcommand{\tu}{\tilde{u}}
\newcommand{\tF}{\tilde{F}}
\newcommand{\debar}{\overline{\de}}
\newcommand{\Z}{\mathbb Z}
\newcommand{\C}{\mathbb C}
\newcommand{\Po}{\mathbb P}
\newcommand{\zbar}{\overline{z}}
\newcommand{\G}{\mathcal{G}}
\newcommand{\So}{\mathcal{S}}
\newcommand{\Ko}{\mathcal{K}}
\newcommand{\U}{\mathcal{U}}
\newcommand{\B}{\mathbb B}
\newcommand{\oB}{\overline{\mathbb B}}
\newcommand{\Cur}{\mathcal D}
\newcommand{\Dis}{\mathcal Dis}
\newcommand{\Levi}{\mathcal L}
\newcommand{\SP}{\mathcal SP}
\newcommand{\Sp}{\mathcal Q}
\newcommand{\A}{\mathcal O^{k+\alpha}(\overline{\mathbb D},\C^n)}
\newcommand{\CA}{\mathcal C^{k+\alpha}(\de{\mathbb D},\C^n)}
\newcommand{\Ma}{\mathcal M}
\newcommand{\Ac}{\mathcal O^{k+\alpha}(\overline{\mathbb D},\C^{n}\times\C^{n-1})}
\newcommand{\Acc}{\mathcal O^{k-1+\alpha}(\overline{\mathbb D},\C)}
\newcommand{\Acr}{\mathcal O^{k+\alpha}(\overline{\mathbb D},\R^{n})}
\newcommand{\Co}{\mathcal C}
\newcommand{\Hol}{{\sf Hol}(\mathbb H, \mathbb C)}
\newcommand{\Aut}{{\sf Aut}(\mathbb D)}
\newcommand{\D}{\mathbb D}
\newcommand{\oD}{\overline{\mathbb D}}
\newcommand{\oX}{\overline{X}}
\newcommand{\loc}{L^1_{\rm{loc}}}
\newcommand{\la}{\langle}
\newcommand{\ra}{\rangle}
\newcommand{\thh}{\tilde{h}}
\newcommand{\N}{\mathbb N}
\newcommand{\kd}{\kappa_D}
\newcommand{\Ha}{\mathbb H}
\newcommand{\ps}{{\sf Psh}}
\newcommand{\Hess}{{\sf Hess}}
\newcommand{\subh}{{\sf subh}}
\newcommand{\harm}{{\sf harm}}
\newcommand{\ph}{{\sf Ph}}
\newcommand{\tl}{\tilde{\lambda}}
\newcommand{\gdot}{\stackrel{\cdot}{g}}
\newcommand{\gddot}{\stackrel{\cdot\cdot}{g}}
\newcommand{\fdot}{\stackrel{\cdot}{f}}
\newcommand{\fddot}{\stackrel{\cdot\cdot}{f}}
\def\v{\varphi}
\def\Re{{\sf Re}\,}
\def\Im{{\sf Im}\,}
\def\rk{{\rm rank\,}}
\def\rg{{\sf rg}\,}

\newtheorem{theorem}{Theorem}[section]
\newtheorem{lemma}[theorem]{Lemma}
\newtheorem{proposition}[theorem]{Proposition}
\newtheorem{corollary}[theorem]{Corollary}

\theoremstyle{definition}
\newtheorem{definition}[theorem]{Definition}
\newtheorem{example}[theorem]{Example}

\theoremstyle{remark}
\newtheorem{remark}[theorem]{Remark}
\numberwithin{equation}{section}

\title[Holomorphic Evolution]{Holomorphic evolution: metamorphosis of the Loewner equation}
\author[F. Bracci]{Filippo Bracci}
\address{Dipartimento Di Matematica\\
Universit\`{a} di Roma \textquotedblleft Tor Vergata\textquotedblright\ \\
Via Della Ricerca Scientifica 1, 00133 \\
Roma, Italy} \email{fbracci@mat.uniroma2.it}
\date\today




\maketitle

\tableofcontents

\section{Introduction}

The theory developed by Ch. Loewner \cite{Loewner} in 1923, and nowadays bearing his name, has been extended and used in the past decades to solve many different problems in the area of complex analysis. Just to name a few instances, the Loewner theory is one of the main tool in the de Branges' proof of the Bieberbach conjecture, and it has been recently successfully exploited in connection with stochastic equations to study scaling limits of various probabilistic and physical models (giving rise to the so called SLEs of Oded Schramm).

The original  idea of Loewner was to
represent a family of domains obtained by removing from the complex plane a Jordan arc  by means of a family (nowadays known as
a {\sl Loewner chain})  of univalent functions defined on the
unit disc and satisfying a suitable differential equation. Such
a machinery was later studied and extended to other types  of
simply connected domains  by  Kufarev in 1943 and Pommerenke in
1965 (\cite{Kuf1943}, \cite{Pommerenke-65}, and see also \cite[Chapter 6]{Pommerenke}).

The classical Loewner partial differential equation in the unit disc
$\D:=\{\zeta\in \C: |\zeta|<1\}$ is given by
\[
\frac{\de f_t(z)}{\de t}=-\frac{\de f_t(z)}{\de z}G(z,t).
\]
where $f_t:\D \to \C$ is a family of univalent mappings depending on the parameter $t\geq 0$, $f_t(0)=0$, $f_t'(0)=e^t$, and $G(w,t)=-w\frac {1+k(t)w}{1-k(t)w}$ for
some continuous function
$k:[0,+\infty)\rightarrow\partial\mathbb{D}$. The vector field $G(w,t)$ is a so-called {\sl Herglotz vector field}.

The classical {\sl radial Loewner equation}  is the following associated
non-autonomous ordinary differential equation
\begin{equation*}
\begin{cases}
\overset{\bullet}{w}  =G(w,t)& \quad\text{for almost every }t\in\lbrack s,\infty)\\
w(s)   =z.&
\end{cases}
\end{equation*}
The solutions $t\mapsto
\v_{s,t}(z)$ of such a differential equation possess certain ``semigroup-type'' properties, and
the family $(\v_{s,t})$ is called
 an {\sl evolution family} of the unit disc.

The relations among the three objects, that is, Loewner's chains, evolution families and Herglotz vector fields, is the core of Loewner's theory and its extensions and generalizations.

The aim of this note is to provide an updated account of the extensions and generalizations of the original Loewner theory, with a particular view toward the geometrical and dynamical aspects of the above equations and their invariant forms. We will start by presenting quite in detail the original work of Loewner, and the extension by Pommerenke, Kufarev and Schramm in the unit disc. Next, we will describe infinitesimal generators of semigroups of holomorphic self-maps on complex manifolds, with the target of presenting a very general  and natural definition of Herglotz vector fields and evolution families, as discovered by the author and M. D. Contreras and S. D\'iaz-Madrigal in \cite{Br-Co-Di-EF1}, \cite{Br-Co-Di-EF2}. In this new framework, the accent is put on the evolution families considered as families of holomorphic self-maps resembling semigroups. Hence, on the one side they are objects that can be iterated, creating a ``dynamical system'', and on the other side they are generated by  non-autonomous vector fields which are semicomplete for almost all times. In this optic, Loewner chains are essentially viewed as ``intertwining mappings'' which conjugated the dynamical behavior of an evolution family on a complex manifold with the geometry of an ``abstract basin of attraction'' which we call the {\sl Loewner range}. The construction of Loewner chains, taken by the work of the author with L. Arosio, H. Hamada, G. Kohr \cite{ABHK}, is categorial and provides the ``PDE Loewner equation'' in its full generality.

Some results are presented with a sketch of the proof, and some efforts are made to relate various branches of the theory to a single unified source (for instance the reverse equation used from Schramm is presented here starting from the Loewner classical equation). However, many applications of Loewner theory and other areas of related researches are not discussed here. For further material, we refer the reader to the survey paper \cite{ABCD}, the books \cite{G-K} (for basic one dimensional and the higher dimensional theory) and \cite{La} (for SLEs), and the recent paper \cite{CDG-annulo, {CDG-annulo2}} (for the multi-connected cases).

\medskip

These notes are the revised version of the text of the plenary talk the author gave at XIX Congress  of the Unione Matematica Italiana held in Bologna, 12-17 September 2011. The author wants to sincerely thank the organizers for both the invitation to give the talk and the possibility of writing these notes. Also, the author warmly thanks Pavel Gumenyuk for several comments which improved this paper.

\section{The classical Loewner equation}\label{sect-class}

We define the ``class $\So$'' as the set of all univalent ({\sl i.e.} holomorphic and injective) function $f:\D \to \C$ such that $f(0)=0$, $f'(0)=1$. Namely,
\[
\So:=\{f:\D \to \C : f \hbox{ is univalent}, f(0)=0, f'(0)=1\}.
\]

By the Riemann mapping theorem, given a simply connected domain $D\subset \C$, $D\neq \C$, there exists a univalent and surjective mapping $g: \D \to D$, called a {\sl Riemann map}. Up to translation we can assume that $0\in D$. With this assumption, one can show that there exists a {\sl unique} Riemann map $g_D:\D \to D$ such that $g_D(0)=0$ and $\lambda:=g_D'(0)>0$. In particular, if we dilate $D$ by $1/\lambda$, the unique Riemann map $g:\D\to \frac{1}{\lambda}D$ such that $g(0)=0$, $g'(0)>0$, belongs to the class $\So$.

\begin{definition}
A {\sl slit map} $f:\D \to \C$ is a univalent mapping such that $f(0)=0$ and the complement of $f(\D)$ in $\C$ is a Jordan arc $\Gamma$, {\sl i.e.},  there exists a continuous  injective curve $\gamma:[0,T)\to\C$ such that $\lim_{t\rightarrow T}|\gamma(t)|=\infty$, $\Gamma:=\gamma([0,T))$ and $f(\D)=\C\setminus \Gamma$.
\end{definition}

One can prove that slit maps are dense (with respect to the topology of uniform convergence on compacta) in the class $\So$. Therefore, if one can prove certain bounds on slit maps, they propagate to all the class $\So$.

Now we are going to discuss the so-called ``parametric representation of slit maps'' due to Ch. Loewner.

Let  $\gamma:[0,T)\to\C$ be a Jordan arc such that $\lim_{t\rightarrow T}|\gamma(t)|=\infty$, $\Gamma:=\gamma([0,T))$. Assume, up to translation, that $0\not\in \Gamma$. Let $\Gamma_t:=\gamma([t,T))$, for $t\in [0,T)$. Then $D_t:=\C\setminus \Gamma_t$ is a family of simply connected domains with the property that $D_s\subset D_t$ for $s<t$.

Let $f_t:\D \to D_t$ be the Riemann map such that $f_t(0)=0, f_t'(0)>0$. Then we can expand $f_t$ and obtain
\[
f_t(z)=\beta(t)[z+b_2(t)z^2+b_3(t)z^3+\ldots].
\]
By the geometry of $D_t$'s, one can show that $t\mapsto \beta(t), b_j(t)$ are continuous for all $j\geq 2$. Moreover, consider  the  map $z\mapsto \v_{s,t}(z):=f_t^{-1}\circ f_s(z)$ ($s<t$). Then $\v_{s,t}:\D \to \D$ is holomorphic, injective, $\v_{s,t}(0)=0$ and it is not the identity map. By the Schwarz' Lemma it follows then $\v_{s,t}'(0)<1$. Hence, $t\mapsto \beta(t)$ is strictly increasing.

It can be  also proved that $\lim_{t\to T}\beta(t)=+\infty$. Indeed, if this is not the case, then thanks to the so-called ``distortion theorems'' for the class $\So$, one would find a converging (in the topology of uniform convergence on compacta) subsequence of $\{f_t\}$ which would converge to a biholomorphism from $\D$ to $\C$.

Therefore, we can re-parameterize the Jordan arc $\gamma$ by define $\sigma(s):=\beta^{-1}(e^s)$ and $\tilde{\gamma}:[0,+\infty)\to \C$ as $\tilde{\gamma}(s)=\gamma(\sigma(s))$. With this new parametrization we have
\[
f_t(z)=e^t\left[z+\sum_{j\geq 2} b_j(t) z^j \right].
\]

\begin{definition}
A family $(f_t)$ of univalent mappings $f_t:\D \to \C$ is called a {\sl classical Loewner chain} if
\begin{enumerate}
  \item $f_t(0)=0, f_t'(0)=e^t$ for all $t\geq 0$ and
  \item $f_s(\D)\subset f_t(\D)$ for all $0\leq s\leq t$.
\end{enumerate}
\end{definition}

Now we can state Loewner's original result:

\begin{theorem}[Loewner]\label{classical-Loewner}
Let $(f_t)$ be a classical Loewner chain of slit maps. Let $\v_{s,t}:=f_t^{-1}\circ f_s: \D \to \D$, $0\leq s\leq t$. Then there exists $k:[0,+\infty)\to \de \D$ a continuous function such that for all $t\in [0,+\infty)$ and $z\in \D$
\begin{equation}\label{classicalLoewner}
\frac{\de \v_{s,t}(z)}{\de t}=-\v_{s,t}(z)\frac{1+k(t)\v_{s,t}(z)}{1-k(t)\v_{s,t}(z)}.
\end{equation}
Moreover, $\lim_{t\to \infty}e^t \v_{s,t}(z)=f_s(z)$ uniformly on compacta.
\end{theorem}

\begin{proof}[Sketch of the proof]
The map $\v_{s,t}:\D \to \D$ is univalent, $\v_{s,t}(0)=0$, $\v_{s,t}'(0)=e^{s-t}$. Note also that
\[
\v_{s,t}=\v_{u,t}\circ \v_{s,u}\quad 0\leq s\leq u\leq t.
\]
Also, by Carath\'eodory's extendability result (see, {\sl e.g.} \cite{Pommerenke}), the map $\v_{s,t}$ is continuous up to $\de \D$ for all $0\leq s\leq t$.

The function $\D\setminus\{0\} \ni z\mapsto g_{s,t}(z):=\frac{\v_{s,t}(z)}{z}$ extends holomorphic in $\D$ by defining $g_{s,t}(0)=\v_{s,t}'(0)=e^{s-t}$ and $g_{s,t}(z)\neq 0$ for all $z\in \D$ because $\v_{s,t}$ is injective and it is equal to zero at $z=0$. Hence, since $\D$ is simply connected, it is possible to define the logarithm $\phi_{s,t}(z):=\log g_{s,t}(z)$, choosing the branch of $\log$ such that  $\log e^{s-t}=s-t$.

Now, the image of $\D$ under $\v_{s,t}$ is $\D\setminus \Gamma_{s,t}$, where $\Gamma_{s,t}$ is a Jordan arc contained in $\D$ and starting from a boundary point. Therefore, there exists an arc $A_{s,t}\subset \de \D$ such that $\v_{s,t}(\de \D \setminus A_{s,t})\subset \de \D$, hence $\Re \phi_{s,t}(\de \D \setminus A_{s,t})=0$. Also, $\Re \phi_{s,t}(z)<0$ for all $z\in \D$.
We apply then the Poisson formula  and obtain
\begin{equation*}
\phi_{s,t}(z)=\frac{1}{2\pi}\int_{A_{s,t}}\frac{e^{i\theta}+z}{e^{i\theta}-z} \Re\phi_{s,t}(e^{i\theta}) d \theta.
\end{equation*}
By the mean value theorem, there exist $\theta', \theta''\in \R$ such that $e^{i\theta'}, e^{i\theta''}\in A_{s,t}$ and
\begin{equation*}
\begin{split}
&\phi_{s,t}(\v_{m,s}(z))=\frac{1}{2\pi}\int_{A_{s,t}}\frac{e^{i\theta}+\v_{m,s}(z)}{e^{i\theta}-\v_{m,s}(z)} \Re\phi_{s,t}(\v_{m,s}(e^{i\theta})) d \theta\\&=
\frac{1}{2\pi}\int_{A_{s,t}}\Re\phi_{s,t}(\v_{m,s}(e^{i\theta})) d \theta\left[\Re\left(\frac{e^{i\theta'}+\v_{m,s}(z)}{e^{i\theta'}-\v_{m,s}(z)}\right)+i\Im \left(\frac{e^{i\theta''}+\v_{m,s}(z)}{e^{i\theta''}-\v_{m,s}(z)}\right) \right].
\end{split}
\end{equation*}
But
\[
\frac{1}{2\pi}\int_{A_{s,t}}\Re\phi_{s,t}(\v_{m,s}(e^{i\theta})) d \theta=\Re \phi_{s,t}(\v_{m,s}(0))=\Re \phi_{s,t}(0)=s-t.
\]
Therefore
\begin{equation}\label{mez}
\phi_{s,t}(\v_{m,s}(z))=(s-t)\left[\Re\left(\frac{e^{i\theta'}+\v_{m,s}(z)}{e^{i\theta'}-\v_{m,s}(z)}\right)+i\Im \left(\frac{e^{i\theta''}+\v_{m,s}(z)}{e^{i\theta''}-\v_{m,s}(z)}\right) \right].
\end{equation}
Now, $\v_{s,t}(\v_{m,s}(z))=\v_{m,t}(z)$ for $0\leq m\leq s\leq t$, hence
\[
\phi_{s,t}(\v_{m,s}(z))=\log \left( \frac{\v_{s,t}(\v_{m,s}(z))}{\v_{m,s}(z)}\right)=\log \frac{\v_{m,t}(z)}{\v_{m,s}(z)}.
\]
It can be proved that, as $t\to s$, the arc $A_{s,t}$ shrinks to a point $\lambda(s)\in \de \D$ which represents the preimage of the tip of the arc $\Gamma_{s,t}$ under $\v_{s,t}$. From this, and from \eqref{mez} we obtain
\begin{equation*}\begin{split}
\lim_{t\to s} \frac{1}{t-s}\log \frac{\v_{m,t}(z)}{\v_{m,s}(z)}&=-\lim_{t\to s} \left[\Re\left(\frac{e^{i\theta'}+\v_{m,s}(z)}{e^{i\theta'}-\v_{m,s}(z)}\right)+i\Im \left(\frac{e^{i\theta''}+\v_{m,s}(z)}{e^{i\theta''}-\v_{m,s}(z)}\right) \right]\\
&=-\frac{\lambda(s)-\v_{m,t}(z)}{\lambda(s)+\v_{m,t}(z)}.
\end{split}\end{equation*}
Unwrapping the left hand side, we obtain \eqref{classicalLoewner} with $k(t):=1/\lambda(t)$. The rest of the statement is technical and we omit it (see, {\sl e.g.} \cite{Duren}).
\end{proof}

In the proof of Loewner's equation we defined a family $(\v_{s,t})$ of univalent self-mappings of the unit discs having certain semigroup properties. Abstracting those properties we give the following

\begin{definition}\label{def-class-ev}
A family $(\v_{s,t})$ with $0\leq s\leq t<+\infty$ of univalent self-maps of the unit disc $\D$ is a {\sl classical evolution family} if
\begin{enumerate}
  \item $\v_{s,t}=\v_{u,t}\circ \v_{s,u}$ for all $0\leq s\leq u\leq t$,
  \item $\v_{s,t}(0)=0, \v_{s,t}'(0)=e^{s-t}$.
\end{enumerate}
\end{definition}

Given a classical Loewner chain $(f_t)$, it is possible to define a classical evolution family $(\v_{s,t})$ by means of the formula
\begin{equation}\label{functional}
f_t\circ \v_{s,t}=f_s.
\end{equation}

Since $\v_{s,t}=f_t^{-1}\circ f_s$, it is clear that such an evolution family is uniquely determined. However, the converse is not so immediate. We will discuss later (see Theorem \ref{ABHKthm}) of the uniqueness of Loewner chains associated to a given evolution family on complete hyperbolic manifolds, and the reader can easily check that in the classical case the uniqueness follows as a result of the normalization chosen in the definition of classical Loewner chains.

Equation \eqref{classicalLoewner} can be re-written in the following way. Let $k:[0,+\infty)\to \de \D$ be continuous and let
\[
p(z,t):=\frac{1+k(t)z}{1-k(t)z}.
\]
Then $\Re p(z,t)\geq 0$ for all $z\in \D$ and $t\in [0,+\infty)$. Let
\[
G(z,t)=-zp(z,t).
\]

Loewner's equation \eqref{classicalLoewner} reads as
\begin{equation}\label{class-ode}
\frac{\de \v_{s,t}(z)}{\de t}=G(\v_{s,t}(z),t).
\end{equation}

The equation \eqref{classicalLoewner} (or the more general equation \eqref{class-ode}) are known as {\sl radial Loewner equations}.

Looking abstractly to the properties of $G$, we give the following definition

\begin{definition}\label{def-class-Herglotz}
A {\sl classical Herglotz vector field} $G(z,t)=-zp(z,t)$ is a non-autonomous vector field such that
\begin{enumerate}
  \item $[0,+\infty)\ni t\mapsto p(z,t)$ is measurable for all $z\in \D$,
  \item $z\mapsto p(z,t)$ is holomorphic for all $t\in [0,+\infty)$
  \item $\Re p(z,t)\geq 0$ for almost all $t\in [0,+\infty)$,
  \item $p(0,t)=1$ for all $t\in [0,+\infty)$.
\end{enumerate}
\end{definition}

Differentiating \eqref{functional} and taking into account \eqref{class-ode} we obtain the following PDE:
\begin{equation}\label{PDE-class}
\frac{\de f_t(z)}{\de t}=-\frac{\de f_t(z)}{\de z}G(z,t).
\end{equation}

Ch. Pommerenke \cite{Pommerenke, Pommerenke-65} showed that \eqref{functional}, \eqref{class-ode}  and \eqref{PDE-class} hold in the context of classical Loewner chains, classical evolution families and classical Herglotz vector fields, and not only for slit maps. We will discuss later of a more general version of these results.

\subsection{The Bieberbach conjecture}

The Bieberbach conjecture states the following. Let $f\in \So$. Expand $f(z)=z+\sum_{j\geq 2} a_m z^m$. Then
\[
|a_m|\leq m\quad \forall m\in \N.
\]

The Bieberbach conjecture has been positively solved by L. de Branges \cite{dB}, who proved the so called Milin's conjecture (which implies the Bieberbach conjecture) using special functions and Loewner's equation  (a simplified proof is given by C. FitzGerald and Ch. Pommerenke \cite{FP}).

The case $m=2$ is a consequence of the so called ``area theorem'' (see, {\sl e.g.}, \cite{Duren}). The case $m=3$ was proved by Loewner himself, using his equation \eqref{PDE-class}. We give here a brief sketch of his idea. We start with
\begin{equation}\label{Low-PDE}
    \frac{\de f_t(z)}{\de t}=z\frac{\de f_t(z)}{\de z}p(z,t).
\end{equation}
Now, we expand
\[
p(z,t)=1+p_1(t)z+p_2(t)z^2+\ldots,
\]
and also, we write
\[
f_t(z)=e^tz+a_2(t)z^2+\ldots.
\]
Substituting in \eqref{Low-PDE} and equating coefficients with the same degree in $z$, we obtain for almost all $t\geq 0$,
\[
\frac{d}{dt} a_m(t)=\sum_{k=1}^{m-1} k a_k(t) p_{m-k}(t)+ma_m(t).
\]
Now, multiplying both sides by $e^{-mt}$ and integrating, one obtains an expression for $a_m$ which involves the terms $p_k$. By the distortion theorems for $p(z,t)$, it follows that $|p_k(t)|\leq 2$ for all $k$. From here (after some algebraic manipulations which are working well only for $m=2,3$), we obtain the estimates.

\subsection{Slit mappings and the Loewner differential equation} Given a continuous function $k:[0,+\infty)\to \de \D$, one can consider the PDE
\begin{equation}\label{gen-slit}
\frac{\de f_t(z)}{\de t}=z\frac{\de f_t(z)}{\de z}\frac{1+k(t)z}{1-k(t)z},
\end{equation}
with $z\in \D$. The function $k$ is called the {\sl driving term} of the equation.

Loewner's theorem \ref{classical-Loewner} shows that any evolution family of slit mappings satisfies \eqref{gen-slit} with a continuous driving term. The converse is not true: P. P. Kufarev \cite{Kuf1947} showed that the solutions to \eqref{gen-slit} with continuous driving term are not slit mappings in general.

The question is then which are the relations between the properties of the driving term $k$ in \eqref{gen-slit} and the family generated by the solutions. It is known that if the evolution derives from a slit which is real analytic, then $k$ is real analytic. A proof of this fact can be found in \cite{EE}, where C. Earle and A. Epstein proved also that if the slit is of class $C^m$ then the driving term is at least of class $C^{m-1}$.

In \cite{Ma-Ro}, D. Marshall and S. Rohde proved that if the slit in $\C$  is a ``quasiarc'' (namely it is the image of $[0,\infty)$ under a quasiconformal homeomorphism of $\C$) then the driving term is Lipschitz continuous with exponent $1/2$. And conversely,  there exists a constant $C>0$ such that if $|k(t)-k(s)|<C|t-s|^{1/2}$ for all $s,t\in [0,+\infty)$ then $\C\setminus f_t(\D)$ is a quasiarc for all $t$. In Kufarev's example, the driving term $k$ is Lipschitz continuous with exponent $3\sqrt{2}$, thus $C\leq 3\sqrt{2}$. J. Lind \cite{Li} has proved that the best constant $C$ is $4$. However, W. Kager, B. Nienhuis, L. P. Kadanoff \cite{KNK} showed that there exist examples of slit evolutions for which the associated driving terms have arbitrary big norm. In \cite{Pro-Vas}, D. Prokhorov and A. Vasil'ev  extended such results to the case of evolutions of so-called ``chordal type'' (see below) when the slit is an arc tangent to the boundary, proved that in such a case the driving term is $1/3$-Lipschitz.

\section{Kufarev-Loewner chordal equation}

In 1946 P. P. Kufarev \cite{Kuf1946} (developed later by  Kufarev himself, Sobolev and Sporysheva \cite{KSS}) proposed an equation of evolution in the upper
half-plane $\Ha:=\{z\in \C: \Im z>0\}$ analogous to the one proposed by Loewner in the unit
disc. Note that $\Ha$ and $\D$ are conformally equivalent by means of the Cayley transform
\[
\D \ni z\mapsto i\frac{1+z}{1-z}\in \Ha.
\]
However, Kufarev's equation is not just a transliteration from $\D$ to $\Ha$ by means of the Cayley transform of Loewner's evolution equation. Kufarev in fact considered a different equation, where the base point of the evolution is at $\infty$.
This process is connected to  physical problems in
hydrodynamics.

In order to introduce Kufarev's equation properly, let us fix
some notations. Let $\gamma$ be a Jordan arc in the upper
half-plane $\Ha$ with starting point $\gamma(0)=0$. Then, there
exists a unique Riemann map $f_t:\Ha\to\Ha\setminus
\gamma[0,t]$ normalized such that
$$
f_t(z)=z+\frac{c(t)}{z}+O\left(\frac{1}{z^2}\right).
$$ Up to a re-parametrization of the curve $\gamma$, one can assume that $c(t)=-2t$. Such a normalization is sometimes called the {\sl hydrodynamics normalization}.
Under this normalization, one can show that $f_t$ satisfies the
following differential equation. For all $t\geq 0$ and for all
$z\in\Ha$
\begin{equation}\label{hydro-Low}
\frac{\partial f_t(z)}{\partial t}=\frac{-2}{f_t(z)-k(t)}, \qquad f_0(z)=z,
\end{equation}
where $k:[0,+\infty)\to \R$ is a continuous  function. Conversely,
starting with a continuous function $k:[0,+\infty)\to\R$, one
can consider the non-autonomous holomorphic vector field
\[
P(w,t):=\frac{-2}{w-k(t)},
\]
and the associated initial value problem for each
$z\in\Ha$:
\begin{equation*}
\frac{d w}{dt}=P(w(t),t), \qquad w(0)=z.
\end{equation*}
Let $t\mapsto w^z(t)$ denote the only solution of the previous
system, and let $f_t(z):=w^z(t)$. Then $f_t:\Ha \to \Ha$ is univalent. This equation is nowadays known  as the {\sl chordal
Loewner differential equation} and the function $k$ is its
driving term. The name ``chordal'' is due to the picture that the images of
the solutions of the associated characteristic equation
draw when taking the time-limit: something like the half-plane
erased a chord joining two boundary points.

Moving back to the unit disc by means of the Cayley transform, it is easy to see that the
chordal Kufarev-Loewner equation takes the form
\begin{equation}\label{chordal}
\frac{dz}{dt}=(1-z)^2p(z,t), \qquad z(0)=z,
\end{equation}
where $\Re p(z, t)\geq 0$ for all $t\geq 0$ and $z\in \D$.

We will show later on that both the classical Loewner equation and its radial generalizations and the Kufarev equation are just particular cases of a more general picture.

D. Marshall kindly told me in a private conversation that the classical Loewner equation and the Kufarev one are equivalent in the sense that can be obtained one from the other by means of a suitable construction.

\section{Reversing evolution and SLE's}

The original Loewner equation (and the generalized Pommerenke's and Kufarev's equations) deals with families of univalent mappings from $\D$ to  increasing families of simply connected domains. In the applications it is sometimes useful to consider a reverse evolution. Namely, let $(D_t)_{t\geq 0}$ be a family of simply connected domains contained in the unit disc $\D$ and such that $D_t\subseteq D_s$ for all $0\leq s\leq t$. We also assume that $D_0=\D$.

A typical example is given by considering a Jordan arc $\gamma: [0,+\infty)\to \C$ such that $\gamma(0)\in \de \D$ and $\gamma((0,\infty))\subset \D$. In such a case $D_t=\D\setminus\gamma([0,t])$. If $0\in D_t$ for all $t$, one can consider a chain of univalent mappings $f_t: \D \to D_t$ normalized so that $f_t(0)=0$ and $f_t'(0)>0$. This is a sort of ``reverse classical Loewner evolution''.

Similarly, one can consider a ``reverse chordal Kufarev-Loewner evolution'', taking the upper half-plane model $\Ha$ and removing a growing Jordan arc $\gamma: (0,+\infty)\to \Ha$ such that $\gamma(0)=0$, considering the chain given by $f_t: \Ha \to \Ha \setminus \gamma([0,t])$ with the hydrodynamics normalization.

In this section we restrict ourselves to the case of the ``reverse classical Loewner evolution''. However, one can show the same procedure works for all generalizations of the classical Loewner equation (see also \cite{Co-Di-Gu2}).

For  $t\geq 0$, let  $f_t:\D \to D_t$ be a Riemann mapping normalized such that $f_t(0)=0$ and $f_t'(0)>0$. Assume that $D_t:=f_t(\D)$ be such that $D_t\subset D_s$ for $0\leq s\leq t$ and $f_0={\sf id}$. Moreover, $\D\setminus D_t$ is a Jordan arc $\gamma: [0,+\infty)\to \C$ such that $\gamma(0)\in \de \D$ and $\gamma((0,\infty))\subset \D$.

Let $\beta(t):=f_t'(0)$. By Schwarz' lemma, $\beta:[0,+\infty)\to \R^+$ is a decreasing function, $\beta(0)=1$. Let
\[
A:=\lim_{t\to \infty} \beta(t).
\]
Then $0\leq A<1$. Let $\sigma(s):=\beta^{-1}(e^{-s})$, for $s\in [0,-\ln A)$. We re-parameterize the Jordan arc as $\tilde{\gamma}(s):=\gamma(\sigma(s))$, for $s\in [0,-\ln A)$. With such a parametrization we have $f_t(0)=0, f_t'(0)=e^{-t}$ for $t\in [0,-\ln A)$.

Fix $T\in (0,-\ln A)$, and, for $t\in [0,T]$, let $\phi_t(z):=f_{T-t}(z)$. By definition, $\phi_s(\D)\subset \phi_t(\D)$ for all $0\leq s\leq t\leq T$ and $\phi_T={\sf id}$. Also, $\phi_t(0)=0, \phi_t'(0)=e^{t-T}$.

The family $(\v_t)$ is a ``classical Loewner chain'' as defined in Section \ref{sect-class}, except that $t\in [0,T]$ instead of taking values in $[0,+\infty)$. In any case, we can define the associated ``evolution family'' $\v_{s,t}:=f_t^{-1}\circ f_s$. It is easy to check that such a family satisfies the requirement of Definition \ref{def-class-ev}, taking $0\leq s\leq t\leq T$.

Then Theorem \ref{classical-Loewner} applies and  there exists $k_T:[0,T)\to \de \D$ a continuous function such that for all $t\in [0,T)$ and $z\in \D$ equation \eqref{classicalLoewner} holds with $k_T$ replacing $k$.

Then the family $(\phi_t)_{t\in [0,T]}$ satisfies  \eqref{PDE-class}, {\sl i.e.}
\[
\frac{\de \phi_t(z)}{\de t}=\phi_t'(z) z \frac{1+k_T(t)z}{1-k_T(t)z}.
\]
Now, taking into account that $\phi_t=f_{T-t}$, from the previous equation we obtain
\[
\frac{\de f_t(z)}{\de t}=-\frac{\de f_t(z)}{\de z} z \frac{1+k_T(T-t)z}{1-k_T(T-t)z}, \quad t\in [0,T).
\]
This yields that $k_T(T-t)=k_{T'}(T'-t)$ for all $t\leq \min\{T, T'\}$. Thus, setting $k(t):=k_T(T-t)$ whenever $t\in [0, T)$, we find
\begin{equation}\label{reverse}
\frac{\de f_t(z)}{\de t}=-\frac{\de f_t(z)}{\de z} z \frac{1+k(t)z}{1-k(t)z}, \quad t\in [0,-\ln A).
\end{equation}

Note that  \eqref{reverse} differs from \eqref{PDE-class} by a sign. Now, let
\[
g_t:=f_t^{-1}: D_t \to \D.
\]
Since $z=f_t(g_t(z))$ for all $z\in f_t(\D)=D_t$, differentiating in $t$ and taking into account \eqref{reverse}, we obtain
\[
\begin{split}
0&=\frac{\de f_t}{\de t}(g_t(z))+\frac{\de f_t}{\de z}(g_t(z))\frac{\de g_t(z)}{\de t}\\&=
-\frac{\de f_t}{\de z}(g_t(z)) z \frac{1+k(t)g_t(z)}{1-k(t)g_t(z)}+\frac{\de f_t}{\de z}(g_t(z))\frac{\de g_t(z)}{\de t}.
\end{split}
\]
Since $f_t$ is univalent, we get
\begin{equation}\label{reverse-inverse}
\frac{\de g_t(z)}{\de t}=z \frac{1+k(t)g_t(z)}{1-k(t)g_t(z)}
\end{equation}
for all $t\in [0, -\ln A)$ and $z\in D_t$.

Note that, given $z\in \D$, \eqref{reverse-inverse} holds for all $t\in [0,-\ln A)$ if $z\not\in \tilde{\gamma}([0,-\ln A))$. If $z=\tilde{\gamma}(t_0)$ then \eqref{reverse-inverse} holds for all $t\in [0, t_0)$.

Putting together the previous considerations we have

\begin{theorem}[Reverse classical radial Loewner evolution]
Let $\gamma:[0,M)\to \C$ be a Jordan arc such that $\gamma(0)\in \de \D$ and $\gamma((0,M))\subset \D$. Let $f_t:\D \to \D$, for $t\in [0,M)$, be a family of Riemann mappings such that $\D\setminus f_t(\D)=\gamma((0,t])$, $f_t(0)=0$ and $f_t'(0)=e^{-t}$. Let $g_t:=f_t^{-1}$. Then there exists a continuous function $k:[0, M)\to \de\D$ such that
\[
\frac{\de g_t(z)}{\de t}=z \frac{1+k(t)g_t(z)}{1-k(t)g_t(z)}
\]
for all $z\in \D$ and either for all $t\in [0, M)$ if $z\not\in \gamma([0,M))$, or for all $t\in [0,t_0)$ if $z=\gamma(t_0)$.
\end{theorem}

A similar argument as before applies to the Kufarev-Loewner chordal equation. In particular, we obtain the reverse evolution from \eqref{hydro-Low}. That is

\begin{theorem}[Reverse classical chordal Kufarev-Loewner evolution]
Let $\gamma:[0,M)\to \C$ be a Jordan arc such that $\gamma(0)=0$ and $\gamma((0,M))\subset \Ha$. Let $f_t:\Ha \to \Ha$, for $t\in [0,M)$, be a family of Riemann mappings such that $\Ha\setminus f_t(\Ha)=\gamma((0,t])$, $f_t(z)=z-\frac{2t}{z}+O\left(\frac{1}{z^2}\right)$. Let $g_t:=f_t^{-1}$. Then there exists a continuous function $k:[0, M)\to \R$ such that
\[
\frac{\partial g_t(z)}{\partial t}=\frac{2}{g_t(z)-k(t)},
\]
for all $z\in \Ha$ and either for all $t\in [0, M)$ if $z\not\in \gamma([0,M))$, or for all $t\in [0,t_0)$ if $z=\gamma(t_0)$.
\end{theorem}

\subsection{The Schramm-Loewner equation}

In 1999 Oded
Schramm [S] had the wonderful idea of replacing the  driving term of the classical Loewner equation
for single-slit maps with a weighted Brownian motion, inventing
the nowadays well known stochastic-Loewner equations, or
Schramm-Loewner's equations. In particular, the {\sl (chordal)
stochastic Loewner evolution} with parameter $k\geq 0$
(SLE$_k$) starting at a point $x\in \mathbb R$ is the random
family of univalent maps $(g_t)$ obtained from the reverse classical chordal Kufarev-Loewner equation replacing  the driving term $k(t)$ with $\sqrt{k} B_t$, where $B_t$ is a standard one
dimensional Brownian motion such that $\sqrt{k}B_0=x$. That is
\begin{equation*}
\frac{\partial g_t(z)}{\partial t}=\frac{2}{g_t(z)-\sqrt{k} B_t}, \qquad g_0(z)=z.
\end{equation*}

Similarly, one can define a radial stochastic Loewner
evolution starting from the reverse classical  radial Loewner equation replacing  the driving term $k(t)$ with $e^{-i\sqrt{k} B_t}$, {\sl i.e.}
\[
\frac{\de g_t(z)}{\de t}=z \frac{1+e^{-i\sqrt{k} B_t}g_t(z)}{1-e^{-i\sqrt{k} B_t}g_t(z)}, \qquad g_0(z)=z.
\]

The SLE$_k$ depends on the choice of the Brownian motion and it
comes in several flavours depending on the type of Brownian
motion exploited. For example, it might start at a fixed point
or start at a uniformly distributed point, or might have a
built in drift and so on. The parameter $k$ controls the rate
of diffusion of the Brownian motion and the behaviour of the
SLE$_k$ critically depends on the value of~$k$.

The SLE$_2$ corresponds to the loop-erased random walk and the
uniform spanning tree. The SLE$_{8/3}$ is conjectured to be the
scaling limit of self-avoiding random walks. The SLE$_3$ is
conjectured to be the limit of interfaces for the Ising model,
while the SLE$_4$ corresponds to the harmonic explorer and the
Gaussian free field. The SLE$_6$ was used by Lawler, Schramm
and Werner in 2001 [LSW1], [LSW2] to prove the conjecture of
Mandelbrot (1982) that the boundary of planar Brownian motion
has fractal dimension $4/3$. Moreover, Smirnov [Sm] proved that
SLE$_6$ is the scaling limit of critical site percolation on
the triangular lattice. This result follows from his celebrated
proof of Cardy's formula. We refer the reader to the very beautiful book of G. Lawler \cite{La} for more details.

\section{Semigroups and infinitesimal generators}

Looking at the classical radial Loewner equation \eqref{classicalLoewner} and the classical chordal Kufarev-Loewner equation \eqref{chordal}, one notices that there is a similitude between the two. Indeed, we can write both the equation in the form
\[
\frac{\de z(t)}{\de t}=G(z,t),
\]
with
\[
G(z,t)=(\tau-z)(1-\overline{\tau}z)p(z,t),
\]
where $\tau=0, 1$ and  $\Re p(z,t)>0$ for all $z\in \D$ and $t\geq 0$.

The reason for the previous formula is not at all by chance, but it reflects a very important feature of ``Herglotz vector fields''. In order to give a rough idea of what we are aiming, consider the case $\tau=0$ (the radial case). Fix $t=t_0\in [0,+\infty)$. Consider the holomorphic vector field $H(z):=G(z,t_0)$. Let $h(z):=|z|^2$. Then,
\begin{equation}\label{stringo}
dh_z(H(z))=2\Re \la H(z), z\ra =-|z|^2\Re p(z,t_0) \leq 0, \quad \forall z\in \D.
\end{equation}
This Lyapunov type inequality has a deep geometrical meaning. Indeed, \eqref{stringo} tells that $H$ points toward the center of the level sets of $h$, which are concentric circles centered at $0$. For each $z_0\in \D$, consider then the Cauchy problem
\begin{equation}\label{Cauchy}
\begin{cases}
\frac{d w(t)}{dt}=H(w(t)),\\
w(0)=z_0
\end{cases}
\end{equation}
and let $w^{z_0}:[0,\delta)\to \D$ be the maximal solution (such a solution can propagate also in the ``past'', but we just consider the ``future'' time). Since $H$ points inward with respect to all circles centered at $0$, the flow $t\mapsto w^{z_0}(t)$ cannot escape from the circle $h(z)=h(z_0)$. Therefore, the flow is defined for all future times, namely, $\delta=+\infty$. This holds for all $z_0\in \D$.

Hence, the Herglotz vector field $G(z,t_0)$ has the feature to be $\R^+$-semicomplete for all fixed $t_0$.

Let $H$ be a $\R^+$-semicomplete holomorphic vector field on $\D$ and let $[0,+\infty)\ni\mapsto w^z(t)\in \D$ be the solution of \eqref{Cauchy}.
By the holomorphic flow-box theorem, the map
\[
[0,+\infty)\times \D \mapsto \phi_t(z):=w^{z}(t)\in \D
\]
is real analytic, and for all fixed $t$, the map $z\mapsto \phi_t(z)$ is holomorphic. By definition $\phi_0={\sf id}$ and by the uniqueness of solutions of \eqref{Cauchy},
\[
\phi_{t+s}=\phi_t \circ \phi_s=\phi_s\circ \phi_t \quad \forall s,t\geq 0.
\]
In other words, $(\phi_t)$ is a continuous morphism of semigroups between $(\R_{\geq 0}, +)$ endowed with the Euclidean topology and the semigroup of holomorphic self-maps of the unit disc $({\sf Hol}(\D,\D), \circ)$ endowed with the topology of uniform convergence on compacta.

Recall that a holomorphic vector field $H$ on a complex manifold $M$ is a section of the holomorphic tangent bundle $TM$. In case $M$ is a domain in $\C^n$ (for instance the unit disc in $\C$), then $TM\simeq M\times \C^n$ and thus we can interpret $H$ as a holomorphic function from $M$ to $\C^n$. With this in mind, we give the following definition:

\begin{definition}
Let $M$ be a complex manifold. A holomorphic vector field $H: M \to TM$ is  an {\sl infinitesimal generator} if for each $z_0\in M$, its flow starting from $z_0$ is defined for all $t\geq 0$.
\end{definition}

Also, we define:

\begin{definition}
Let $M$ be a complex manifold. A {\sl continuous semigroup of holomorphic self-maps of $M$}, $(\phi_t)_{t\geq 0}$, is a continuous morphism of semigroups between $(\R_{\geq 0}, +)$ endowed with the Euclidean topology and the semigroup of holomorphic self-maps of $M$  endowed with the topology of uniform convergence on compacta.
\end{definition}

It can be shown that if $(\phi_t)$ is a semigroup, then for each $t\geq 0$, the map $z\mapsto \phi_t(z)$ is univalent.

Let $M$ be a complex manifold. By the holomorphic flow-box theorem, for an infinitesimal generator $H$ on $M$ there exists a unique continuous semigroups $(\phi_t)_{t\geq 0}$ of holomorphic self-maps of $M$ such that
\begin{equation}\label{semigruppo}
\frac{\de \phi_t(z)}{\de t}=H(\phi_t(z)).
\end{equation}
Conversely, given a continuous semigroup of holomorphic self-maps of $M$, there exists a unique infinitesimal generator $H$ on $M$ such that \eqref{semigruppo} holds (\cite{Berkson-Porta}, see also {\sl e.g.}, \cite{Abate}, \cite{Reich-Shoikhet}).

Much has been done in the theory of semigroups, see \cite{Reich-Shoikhet} for a very good recent account. Here we content ourselves to examine the theory we need for our aim.

As we saw before, a classical Herglotz vector field $G(z,t)$ (as defined in Definition \ref{def-class-Herglotz}) has the property that for all $t\geq 0$, the holomorphic vector field $z\mapsto G(z,t)$ is an infinitesimal generator, and one might suspect that this is the right choice for a workable definition of a general Herglotz vector field.  Therefore, it is fundamental to characterize which holomorphic vector fields are infinitesimal generators.

The previous argument  with the function $h$, gives a basic rough idea of the way one can characterize infinitesimal generators.
Before going ahead, we need to recall some few facts about the so-called ``invariant distances''. We refer the reader to \cite{Kob} and \cite{Abate} for details.

Let $\omega:\D\times \D \to \R^+$ denote the Poincar\'e distance on $\D$. Recall that
\[
\omega(z,w)=\frac{1}{2}\log \frac{1+|T_z(w)|}{1-|T_z(w)|},
\]
where $T_z(w)=\frac{z-w}{1-\overline{z}w}$. Essentially by a re-interpretation of the Schwarz lemma, the Poincar\'e distance has the property of being shrunk by holomorphic self-maps of the unit disc, namely, if $f:\D \to \D$ is holomorphic, then
\[
\omega(f(z), f(w))\leq \omega(z,w)\quad \forall z,w\in \D.
\]
Moreover, there is equality for some $z\neq w$---and hence for all---if and only if $f$ is an automorphism of $\D$. The Poincar\'e distance is of class $C^\infty$ outside the diagonal. Take a semigroup $(\phi_t)$ of holomorphic self-maps of $\D$ generated by the infinitesimal generator $H$. Then, for $z\neq w$, the function
\begin{equation}\label{function1}
t\mapsto \omega(\phi_t(z), \phi_t(w))
\end{equation}
is differentiable (because $\phi_t(z)\neq \phi_t(w)$ for all $t\geq 0$ being the map injective) and decreasing, since
\[
\begin{split}
\omega(\phi_t(z), \phi_t(w))&=\omega(\phi_{t-\epsilon+\epsilon}(z), \phi_{t-\epsilon+\epsilon}(w))
\\&=\omega(\phi_\epsilon(\phi_{t-\epsilon}(z)), \phi_\epsilon(\phi_{t-\epsilon}(w))) \leq \omega(\phi_{t-\epsilon}(z), \phi_{t-\epsilon}(w)).
\end{split}
\]
Differentiating in $t$ at $t=0$ we obtain thus
\begin{equation}\label{contrae-disco}
d\omega_{(z,w)}\cdot(H(z), H(w))\leq 0, \quad \forall z,w\in \D, z\neq w.
\end{equation}

Note that if $\phi_t(0)=0$ for all $t\geq 0$, then $H(0)=0$ and the previous equation for $w=0$ is equivalent to $dh_z (H(z))\leq 0$ (where, as before, $h(z)=|z|^2$), that is,
\begin{equation}\label{Remeno}
\Re \la H(z), z\ra \leq 0.
\end{equation}
However, being $H(0)=0$, it follows that $H(z)=-zp(z)$ for some holomorphic function $p:\D \to \C$, and \eqref{Remeno} implies that $p:\D \to \{w\in \C : \Re w \geq 0\}$.

Condition \eqref{contrae-disco} is also necessary to ensure that $H$ is an infinitesimal generator (see \cite{BCM}). The geometric reason of such is that such equation means that Poincar\'e discs are shrunk by the flow of $H$, hence the flow starting at any given point of $\D$ cannot reach the boundary in a finite time. Analytically, \eqref{contrae-disco} translates saying that the function \eqref{function1} (where $t\mapsto \phi_t(z)$ denotes here  the flow starting at $z$) is decreasing in time, and therefore the vector field $H$ is semicomplete. In fact, starting from \eqref{contrae-disco} one can derive  useful equivalent analytical characterizations of infinitesimal generators in the unit disc---historically, such characterizations have been derived directly without using formula \eqref{contrae-disco}, which was discovered in \cite{BCM}.

\begin{theorem}[Characterization of infinitesimal generators in the unit disc]\label{Berkson-P}
Let $H: \D \to \C$ be holomorphic. Then the following are equivalent:
\begin{enumerate}
  \item $H$ is an infinitesimal generator,
  \item (\cite{BCM}) $d\omega_{(z,w)}\cdot(H(z), H(w))\leq 0, \quad \forall z,w\in \D, z\neq w$,
  \item (\cite{ARS}) there exist $a\in \C$ and $q:\D \to \{w\in \C: \Re w\geq 0\}$ holomorphic such that for all $z \in \D$
  \[H(z)=a-\overline{a}z^2-z q(z),\]
  \item (Berkson-Porta's formula \cite{Berkson-Porta})  there exist $\tau\in \oD$ and $p:\D \to \{w\in \C: \Re w\geq 0\}$ holomorphic  such that for all $z \in \D$
  \[H(z)=(\tau-z)(1-\overline{\tau}z) p(z).\]
\end{enumerate}
\end{theorem}

\begin{proof}[Sketch of the Proof] We already saw the equivalence between (1) and (2). Now, if $H(0)=0$ the previous discussion shows that $H(z)=-zp(z)$ for some holomorphic function $p:\D \to \{w\in \C: \Re w\geq 0\}$. Hence (3) and (4) holds with $a=\tau=0$ and they are equivalent to (1) and (2) in such a case.

Now, using the infinitesimal Poincar\'e metric, given by $ds^2=\frac{|dz|^2}{(1-|z|^2)^2}$, and the shrinking properties of holomorphic self-maps of the unit disc and arguing similarly as above (see also \cite[Thm. 1.4.14]{Abate}), one can show that $H$ is an infinitesimal generator if and only if
\[
\Re[ 2 \overline{z}H(z)+(1-|z|^2)H'(z)]\leq 0 \quad \forall z\in \D.
\]
As a consequence, the set of infinitesimal generators is a real cone with vertex $0$. Now, given $a\in \C$, a direct computation shows that the holomorphic vector field $\D \ni z\mapsto g_a(z)=a-\overline{a}z^2$ is a generator of a group of automorphisms of $\D$. Namely, both $g_a$ and $-g_a$ are infinitesimal generator. Therefore, a holomorphic vector field $H$ is an infinitesimal generator if and only if $H-g_a$ is an infinitesimal generator for all $a\in \C$. Hence, setting $a:=H(0)$, it follows that a holomorphic vector field $H$ is an infinitesimal generator if and only if
\[
\Re \la H(z)-g_a(z), z\ra \leq 0.
\]
This implies that (3) is equivalent to (1) and (2). The equivalence with (4) in case $\tau\neq 0$ relies on dynamical properties of the semigroups which we are not going to discuss in here, and therefore it is omitted.
\end{proof}

\begin{remark}
Berkson-Porta's formula (4) relates the infinitesimal generator $H$ with the dynamical properties of the associated semigroup $(\phi_t)$. In particular, the point $\tau$ is (except in the case of a group of rotation) the attractive fixed point of the semigroup, {\sl i.e.}, $\phi_t(z)\to \tau$ as $t\to\infty$ for all $z\in \D$.
\end{remark}

\subsection{Higher dimension} In higher dimension one can replace the Poincar\'e distance with the Kobayashi distance. First, we recall the definition of Kobayashi distance (see \cite{Kob} for details and properties). Let $M$ be a complex manifold and let $z,w\in M$. A {\sl chain of analytic discs} between $z$ and $w$ is a finite family of holomorphic mappings $f_j:\D \to M$, $j=1,\ldots, m$ and points $t_j\in (0,1)$ such that
\[
f_1(0)=z, f_1(t_1)=f_2(0),\ldots, f_{m-1}(t_{m-1})=f_m(0), f_m(t_m)=w.
\]
We denote by $\mathcal C_{z,w}$ the set of all chains of analytic discs joining $z$ to $w$. Let $L\in \mathcal C_{z,w}$. The {\sl length of $L$}, denoted by $\ell (L)$ is given by
\[
\ell (L):=\sum_{j=1}^m \omega (0, t_j)=\sum_{j=1}^m \frac{1}{2}\log \frac{1+t_j}{1-t_j}.
\]
We define the {\sl Kobayashi (pseudo)distance $k_M(z,w)$} as follows:
\[
k_M(z,w):=\inf_{L\in \mathcal C_{z,w}} \ell(L).
\]
If $M$ is connected, then $k_M(z,w)<+\infty$ for all $z,w\in M$. Moreover, by construction, it satisfies the triangular inequality. However, it might be that $k_M(z,w)=0$ even if $z\neq w$ (a simple example is represented by $M=\C$, where $k_\C\equiv 0$). In the unit disc, $k_\D\equiv \omega$.

\begin{definition}
A complex manifold $M$ is said to be {\sl (Kobayashi) hyperbolic} if $k_M(z,w)>0$ for all $z,w\in M$ such that $z\neq w$. Moreover, $M$ is said {\sl complete hyperbolic} if $k_M$ is complete.
\end{definition}

Important examples of complete hyperbolic manifolds are given by bounded convex domains in $\C^n$.

The main property of the Kobayashi distance is the following: let $M,N$ be two complex manifolds and let $f: M \to N$ be holomorphic. Then for all $z,w\in M$ it holds
\[
k_N(f(z), f(w))\leq k_M(z,w).
\]

It can be proved that if $M$ is complete hyperbolic, then $k_M$ is Lipschitz continuous (see \cite{AB}). If $M$ is a bounded strongly convex domain in $\C^n$ with smooth boundary, L. Lempert (see, {\sl e.g.} \cite{Kob}) proved that the Kobayashi distance is of class $C^\infty$ outside the diagonal. In any case, even if $k_m$ is not smooth, one can consider the differential $dk_M$ as the Dini-derivative of $k_m$, which coincides with the usual differential at almost every point in $M\times M$.

The following characterization of infinitesimal generators is proved for strongly convex domains in \cite{BCM}, and in general in \cite{AB}:

\begin{theorem}\label{autonomo}
Let $M$ be a complete hyperbolic complex manifold and let $H$ be an holomorphic vector field on $M$. Then the following are equivalent.
\begin{enumerate}
\item $H$ is an infinitesimal generator,
\item For all $z, w\in M$ with $z\neq w$
it holds $$(dk_M)_{(z,w)}\cdot (H(z),H(w))\leq 0.$$
\end{enumerate}
\end{theorem}

\section{$L^d$-Herglotz vector fields and Evolution families}

In the previous section we saw that the classical Herglotz vector fields which appear in the Loewner equation have the property to be infinitesimal generators for all fixed times. We will exploit such a fact to define a general family of Herglotz vector fields. As a matter of notation, if $M$ is a complex manifold, we let  $\|\cdot \|$ be a Hermitian metric on $TM$ and $d_M$ the corresponding integrated distance.

\begin{definition}
\label{Her-vec-man} Let $M$ be a complex manifold. A \textit{weak holomorphic vector field of
order $d\geq 1$} on $M$ is a mapping $G:M\times \R^+\to
TM$ with the following properties:
\begin{itemize}
\item[(i)] The mapping $G(z,\cdot)$ is measurable on $\R^+$ for all
$z\in M$.
\item[(ii)] The mapping $G(\cdot,t)$ is a holomorphic vector field on $M$ for all $t\in \R^+$.
\item[(iii)] For any compact set $K\subset M$ and all $T>0$, there exists
a function $C_{K,T}\in L^d([0,T],\mathbb{R}^+)$ such that
$$\|G(z,t)\|\leq C_{K,T}(t),\quad z\in K, \mbox{ a.e.}\ t\in [0,T].$$
\end{itemize}

A \textit{Herglotz vector field of order $d\geq 1$} is a weak holomorphic vector field $G(z,t)$  of order $d$ with the property that
$M\ni z\mapsto G(z,t)$ is an infinitesimal generator
for almost all $t\in [0,+\infty)$.
\end{definition}
\begin{remark}
If $M$ is complete hyperbolic, then  a weak holomorphic vector field $G(z,t)$  of order $d$  if a Herglotz vector field
of order $d$ if and only if
\begin{equation}\label{herglotz_def}
(dk_M)_{(z,w)}\cdot(G(z,t),G(w,t))\leq 0,\quad z,w\in M, z\neq w,\ \mbox{ a.e. }
t\geq 0.
\end{equation}
This is proved in \cite{BCM}  for strongly convex domains, and in \cite{AB} for the general case.
\end{remark}

Using the so-called distortion theorem for holomorphic mappings $p:\D \to \{w\in \C: \Re w> 0\}$ and the Berkson-Porta formula in Theorem \ref{Berkson-P}, it can be proved that a classical Herglotz vector field as in the sense of Definition \ref{def-class-Herglotz} or given as in the Kufarev-Loewner equation \eqref{chordal}, is a Herglotz vector field of order $\infty$ in the unit disc in the sense of the previous definition.

Herglotz vector fields in the unit disc can be
decomposed by means of Herglotz functions (and this the reason
for the name). We begin with the following definition:

\begin{definition}
Let $d\in [1,+\infty]$. A {\sl Herglotz function of order $d$}
is a function $p:\mathbb{D}\times\lbrack0,+\infty
)\mapsto\mathbb{C}$ with the following properties:

\begin{enumerate}
\item For all $z\in\mathbb{D},$ the function $\lbrack0,+\infty
)\ni t \mapsto p(z,t)\in\mathbb{C}$ belongs to
$L_{loc}^{d}([0,+\infty),\mathbb{C})$;

\item For all $t\in\lbrack0,+\infty),$ the function
$\mathbb{D}\ni z \mapsto p(z,t)\in\mathbb{C}$ is holomorphic;

\item For all $z\in\mathbb{D}$ and for all $t\in\lbrack0,+\infty),$ we
have $\Re p(z,t)\geq0.$
\end{enumerate}
\end{definition}

Then we have the following result which, using the Berkson-Porta formula, gives a general form of the classical Herglotz vector fields:

\begin{theorem}\cite{Br-Co-Di-EF1}Let
$\tau:[0,+\infty)\rightarrow\overline {\mathbb{D}}$ be a
measurable function and let $p:\D\times [0,+\infty)\to \C$ be a
Herglotz function of order $d\in [1,+\infty)$. Then the map
$G_{\tau,p}:\mathbb{D}\times\lbrack
0,+\infty)\rightarrow\mathbb{C}$ given by
\begin{equation*}
G_{\tau,p}(z,t)=(z-\tau(t))(\overline{\tau(t)}z-1)p(z,t),
\end{equation*}
for all $z\in\mathbb{D}$ and for all $t\in\lbrack0,+\infty),$
is a Herglotz vector field  of order $d$ on the unit disc.

Conversely, if $G:\D\times [0,+\infty)\to \C$ is a Herglotz
vector field of order $d\in [1,+\infty)$ on the unit disc, then
there exist a measurable function
$\tau:[0,+\infty)\rightarrow\overline {\mathbb{D}}$ and a
Herglotz function $p:\D\times [0,+\infty)\to \C$  of order $d$
such that $G(z,t)=G_{\tau, p}(z,t)$ for almost every $t\in
[0,+\infty)$ and all $z\in \D$.

Moreover, if $\tilde{\tau}:[0,+\infty)\rightarrow\overline
{\mathbb{D}}$ is another measurable function and
$\tilde{p}:\D\times [0,+\infty)\to \C$ is another Herglotz
function of order $d$ such that $G=G_{\tilde{\tau}, \tilde{p}}$
for almost every $t\in [0,+\infty)$ then
$p(z,t)=\tilde{p}(z,t)$ for almost every $t\in [0,+\infty)$ and
all $z\in \D$ and $\tau(t)=\tilde{\tau}(t)$ for almost all
$t\in [0,+\infty)$ such that $G(\cdot, t)\not\equiv 0$.
\end{theorem}

We also give a generalization of the concept of evolution families:

\begin{definition}\label{L^d_EF}
Let $M$ be a complex manifold.
A family $(\v_{s,t})_{0\leq s\leq t}$ of holomorphic
self-mappings of $M$ is  an {\sl
 evolution family of order $d\geq 1$} (or $L^d$-evolution family) if it satisfies the {\sl evolution property}
\begin{equation}\label{evolution_property}
\v_{s,s}={\sf id},\quad \v_{s,t}=\v_{u,t}\circ \v_{s,u},\quad 0\leq s\leq u\leq t,
\end{equation}
 and if for any $T>0$ and for any compact set
  $K\subset\subset M$ there exists a  function $c_{T,K}\in
  L^d([0,T],\R^+)$ such that
  \begin{equation}\label{ck-evd}
d_M(\v_{s,t}(z), \v_{s,u}(z))\leq \int_{u}^t
c_{T,K}(\xi)d \xi, \quad z\in K,\  0\leq s\leq u\leq t\leq T.
  \end{equation}
\end{definition}

\begin{remark}
The Schwarz lemma and distortion estimates imply that a classical evolution family in the sense of Definition \ref{def-class-ev} is an evolution family of order $\infty$ in $\D$.
\end{remark}

A classical evolution family in the sense of Definition \ref{def-class-ev} is, by its very definition, constituted by univalent maps, while this is not required {\sl a priori} in the general definition of evolution family given in  Definition \ref{L^d_EF}. However, it is always a case that an evolution family is made of univalent functions, as the following proposition (cf. \cite[Prop. 2.3]{ABHK}) shows:

\begin{proposition}
Let $d\in[1,+\infty]$.
Let $(\v_{s,t})$ be an $L^d$-evolution family for some $d\geq 1$. Then for all $0\leq s\leq t$ the map $M\ni z\mapsto \v_{s,t}(z)$ is univalent.
\end{proposition}
\begin{proof}
We proceed by contradiction. Suppose  there exists $0<s<t$ and $z\neq w$ in $M$ such that $\v_{s,t}(z)=\v_{s,t}(w).$ Set $r:= \inf \{u\in [s,t]: \v_{s,u}(z)=\v_{s,u}(w)\}.$
By Lemma \cite[Lemma 2]{Br-Co-Di-EF2}, $\lim_{u\to s+}\v_{s,u}={\sf id}$ uniformly on compacta, we have $r>s$. If $u\in (s,r)$, $$\v_{u,r}(\v_{s,u}(z))=\v_{u,r}(\v_{s,u}(w)),$$ and since $\v_{s,u}(z)\neq \v_{s,u}(w),$ the mappings $\v_{u,r}$, $u\in (s,r)$, are not univalent on a fixed relatively compact subset of $M$. But again by \cite[Lemma 2]{Br-Co-Di-EF2}, $\lim_{u\to r-} \v_{u,r}={\sf id}$ uniformly on compacta, which is a contradiction since the identity mapping is univalent.
\end{proof}

\begin{remark}
If $(\phi_t)$ is a continuous semigroup of holomorphic self-maps of a complex manifold $M$, one can define an evolution family $(\v_{s,t})$ by setting
\[
\v_{s,t}:=\phi_{t-s},\quad s\leq t.
\]
It is not too difficult to check that $(\v_{s,t})$ is a $L^\infty$-evolution family in the sense of the above Definition \ref{L^d_EF}.
\end{remark}

The classical Loewner and Kufarev-Loewner equations can be generalized as follows:

\begin{theorem}\label{prel-thm}
Let $M$ be a complete hyperbolic complex manifold. Then for any  Herglotz vector field $G$ of order
$d\in [1,+\infty]$ there exists a unique $L^d$-evolution family
$(\v_{s,t})$  over $M$ such that for all
$z\in M$
  \begin{equation}\label{solve}
\frac{\de \v_{s,t}}{\de t}(z)=G(\v_{s,t}(z),t) \quad
\hbox{a.e.\ } t\in [s,+\infty).
  \end{equation}
Conversely for any  $L^d$-evolution family $(\v_{s,t})$  over $M$ there exists a  Herglotz vector field $G$ of
order $d$ such that \eqref{solve} is satisfied. Moreover,
if $H$ is another weak holomorphic vector field which satisfies
\eqref{solve} then $G(z,t)=H(z,t)$ for all $z\in M$ and almost
every $t\in \R^+$.
\end{theorem}

Equation \eqref{solve} is the bridge between the $L^d$-Herglotz vector fields and $L^d$-evolution families. The result has been proved in \cite{Br-Co-Di-EF1} for the case $M=\D$ the unit disc in $\C^n$. In \cite{Br-Co-Di-EF1} it has been proved to hold for any complete hyperbolic complex manifold $M$ with Kobayashi distance of class $C^1$ outside the diagonal, but the construction given there only allowed to start with evolution families of order $d=+\infty$. Next, in \cite{HKM} the case of $L^d$-evolution families has been proved for the case $M=\B^n$  the unit ball in $\C^n$. Finally, in \cite{AB}, L. Arosio and the author proved Theorem \ref{prel-thm} in complete generality.

The previous equation, especially in the case of the unit ball of $\C^n$ and for the case $d=+\infty$, with evolution families fixing the origin and having some particular first jets at the origin has been studied by many authors, we cite here J.A. Pfaltzgraff \cite{Pf74}, \cite{Pf75}, T. Poreda \cite{Por91}, I. Graham, H. Hamada, G. Kohr \cite{GHK01}, I. Graham, H. Hamada, G. Kohr, M. Kohr \cite{GHK09} (see also \cite{GKP03}).

The strong relation between semigroups and evolution families on the one side and Herglotz vector fields and infinitesimal generators on the other side, is very much reflected by the so-called ``product formula'' in convex domains of S. Reich and D. Shoikhet \cite{RS} (see also \cite{Reich-Shoikhet}), generalized on complete hyperbolic manifold in \cite{AB}. Such a formula can be rephrased as follows: let $G(z,t)$ be a Herglotz vector field on a complete hyperbolic complex manifold $M$. For almost all $t\geq 0$, the holomorphic vector field $M\ni z\mapsto G(z,t)$ is an infinitesimal generator. Let $(\phi_r^t)$ be the associated semigroups of holomorphic self-maps of $M$. Let $(\v_{s,t})$ be the evolution family associated to $G(z,t)$. Then, uniformly on compacta of $M$ it holds
\[
\phi_t^r=\lim_{m\to \infty} \v_{t,t+\frac{r}{m}}^{\circ m}=\lim_{m\to \infty} \underbrace{(\v_{t,t+\frac{r}{m}}\circ \ldots\circ \v_{t,t+\frac{r}{m}})}_m.
\]
Using such a formula for the case of the unit disc $\D$, in \cite{BCD3} it has been proved the following result which gives a description of semigroups-type evolution families:

\begin{theorem}
Let $G(z,t)$ be a $L^d$-Herglotz vector field in $\D$ and let $(\v_{s,t})$ be the associated $L^d$-evolution family. The following are equivalent:
\begin{enumerate}
  \item there exists a function $g\in L^d_{loc}([0,+\infty),\C)$ and an infinitesimal generator $H$ such that $G(z,t)=g(t)H(z)$ for all $z\in \D$ and almost all $t\geq 0$,
  \item $\v_{s,t}\circ \v_{u,v}=\v_{u,v}\circ \v_{s,t}$ for all $0\leq s\leq t$ and $0\leq u\leq v$.
\end{enumerate}
\end{theorem}

\section{Abstract Loewner chains}

In order to end up the picture started with the classical Loewner theory, we should put in the frame also the Loewner chains.

In the unit disc $\D$, the general theory of Loewner chains has been settled by M. D. Contreras, S. D\'iaz-Madrigal and P. Gumenyuk \cite{Co-Di-Gu}, who showed that to each $L^d$-evolution family $(\v_t)$  is associated an (essentially unique) ``$L^d$-Loewner chain'' $(f_t)$ of univalent maps $f_t:\D \to \C$ such that $f_s=f_t\circ \v_{s,t}$ for all $0\leq s\leq t$. Their proof relies on a limiting process similar (although technically more complicated) to the classical case.

From this point of view, if $D\subset\C^n$ is a bounded domain, apparently, it seems natural to define a Loewner chain as a family of univalent mappings $f_t:D \to \C^n$. In fact, in case $D=\B^n$ the unit ball, much effort has been done to show that, given an evolution family $(\v_{s,t})$ on $\B^n$ such that $\v_{s,t}(0)=0$ and $d(\v_{s,t})_0$ has a special form, then there exists an associated Loewner chain. We cite here the contributions of  J.A. Pfaltzgraff \cite{Pf74}, \cite{Pf75}, T. Poreda \cite{Por91}, I. Graham, H. Hamada, G. Kohr \cite{GHK01}, I. Graham, H. Hamada, G. Kohr, M. Kohr \cite{GHK09}, L. Arosio \cite{A10}, M. Voda \cite{Vo}. In the last two mentioned papers, resonances phenomena among the eigenvalues of $d(\v_{s,t})_0$ are taken into account. However, the (very natural) fact that resonances enter into the game, gives a clue that possibly, if one stays with the willing of looking for chains with values in $\C^n$, associated Loewner chains might not always exist.  Not to talk about evolution families on a complex manifold: in such a case, what is the appropriated target domain for Loewner chains?

The previous question, which leads to a very general theory, has been answer in \cite{ABHK}. Interesting and surprisingly enough, regularity conditions--which were basic in the classical theory for assuming the limiting process to converge---do not play any role. In order to explain our results, we give some definition:

\begin{definition}\label{algebraic_EF}
Let $M$ be a complex manifold.
An {\sl algebraic evolution family} is a family $(\v_{s,t})_{0\leq s\leq t}$ of univalent self-mappings of $M$  satisfying the  evolution property \eqref{evolution_property}.
\end{definition}

A $L^d$-evolution family is an algebraic evolution family because all elements of a $L^d$-evolution family are injective as we showed before.

\begin{definition}
Let $M, N$ be two complex manifolds of the same dimension.  A family  $(f_t)_{t\geq 0}$ of holomorphic mappings $f_t:M\to N$ is a {\sl subordination chain}  if for each $0\leq s\leq t$ there
exists a holomorphic mapping $v_{s,t}:M\to M$ such that
$f_s=f_t\circ v_{s,t}$.
A subordination chain $(f_t)$ and an algebraic evolution family $(\v_{s,t})$ are {\sl associated} if
$$ f_s=f_t\circ \v_{s,t},\quad 0\leq s\leq t.$$

An {\sl algebraic Loewner chain} is a  subordination chain such that each mapping $f_t: M\to N$ is univalent.   The {\sl range}  of an algebraic  Loewner chain is defined as
\[
\rg(f_t):=\bigcup_{t\geq 0}f_t(M).
\]
\end{definition}

Note that an algebraic Loewner chain $(f_t)$ has the property that
$$f_s(M)\subset f_t(M),\quad 0\leq s\leq t.$$

We have the following result which relates algebraic evolution families with algebraic Loewner chains, whose proof is essentially based on abstract categorial analysis:

\begin{theorem}\cite{ABHK}\label{ABHKthm}
Let $M$ be a complex manifold. Then any  algebraic evolution family $(\v_{s,t})$ on $M$ admits  an associated algebraic Loewner chain $(f_t\colon M\to N)$.
Moreover if $(g_t\colon M\to Q)$ is a subordination chain associated with $(\v_{s,t})$ then there exist a holomorphic mapping $\Lambda\colon \rg(f_t)\to Q$ such that $$g_t=\Lambda\circ f_t,\quad \forall t\geq 0.$$ The mapping $\Lambda$ is univalent if and only if $(g_t)$ is an algebraic Loewner chain, and in that case $\rg(g_t)=\Lambda(\rg(f_t)).$
\end{theorem}

\begin{proof}
We define an equivalence relation on the product $M\times \mathbb{R}^+$:
$$(x,s)\sim (y,t)\quad\mbox{iff}\quad   \v_{s,u}(x)=\v_{t,u}(y) \mbox{ for $u$ large enough}.$$
and we let $N:=(M\times \mathbb{R}^+)/_\sim$.
Let $\pi\colon M\times \mathbb{R}^+\to N$ be the projection on the quotient, and let  $i_t\colon M\to M\times \mathbb{R}^+$ be the injection $i_t(x)=(x,t)$. Define a family of mappings $(f_t\colon M\to N)$ as $$f_t:= \pi\circ i_t,\quad t\geq 0.$$ Each mapping $f_t$ is injective and by construction $$f_s=f_t\circ\v_{s,t}, \quad 0\leq s\leq t.$$
Thus we have  $f_s(M)\subset f_t(M)$ for $0\leq s\leq t$ and $N=\bigcup_{t\geq 0}f_t(M)$.

Endow the product  $M\times \mathbb{R}^+$ with the product topology, considering on $\R^+$ the discrete topology. Endow $N$ with the quotient topology. Each mapping $f_t$ is continuous and open, hence it is an homeomorphism onto its image and we define a complex structure on $N$ by considering the $M$-valued charts $(f^{-1}_t,f_t(M))$ for all $t\geq 0$.

If $(g_t\colon M\to Q)$ is a subordination chain associated with $(\v_{s,t})$, then the map $\Psi\colon M\times \mathbb{R}^+\to Q$
$$(z,t)\mapsto g_t(z)$$  is compatible with the equivalence relation $\sim$, thus it passes  to the quotient defining a holomorphic mapping $\Lambda\colon N\to Q$ such that $$g_t=\Lambda\circ f_t,\quad t\geq 0.$$
\end{proof}

The previous theorem shows that the range $\rg (f_t)$ of an algebraic Loewner chain $(f_t)$ is uniquely defined up to biholomorphisms. In particular, given an algebraic evolution family $(\v_{s,t})$ one can define its {\sl Loewner range} ${\sf Lr}(\v_{s,t})$  as the class of biholomorphism of the range of any associated algebraic Loewner chain.

As one can suspect, the $L^d$ regularity of an algebraic evolution family passes to the associated algebraic Loewner chain. This is the right definition:

\begin{definition}
Let $d\in [1,+\infty]$. Let $M, N$ be two complex manifolds of
the same dimension. Let $d_N$ be the distance induced by a Hermitian metric on $N$. An algebraic Loewner chain $(f_t\colon M\to N)$  is a {\sl $L^d$-Loewner chain} (for $d\in [1,+\infty]$) if for any compact set $K\subset\subset M$ and any
  $T>0$ there exists a $k_{K,T}\in L^d([0,T], \R^+)$ such that
 \begin{equation}\label{LCdef}
d_N(f_s(z), f_t(z))\leq \int_{s}^t k_{K,T}(\xi)d\xi
 \end{equation}
for all $z\in K$ and for all $0\leq s\leq t\leq T$.
\end{definition}

The $L^d$-regularity passes from evolution family to Loewner chains and back:

\begin{theorem}\cite{ABHK}
Let  $M$ be a  complete hyperbolic manifold with a given Hermitian metric and $d\in [1,+\infty]$. Let $\v_{s,t}$ be an algebraic evolution family on $M$ and let $(f_t\colon M\to N)$ be an associated algebraic Loewner
chain.
Then $(\v_{s,t})$ is a $L^d$-evolution family on $M$ if and only if $(f_t)$ is a $L^d$-Loewner chain.
\end{theorem}

As one can imagine, once the general Loewner equation is established and Loewner chains have been well defined, even the Loewner-Kufarev PDE can be generalized. We state here a result in this sense from \cite{ABHK}:

\begin{theorem}
Let $M$ be a complete hyperbolic complex manifold, and let $N$ be a complex manifold of the same dimension.
Let $G:M\times \R^+\to TM$ be a Herglotz vector
field of order $d\in[1,+\infty]$
associated with the $L^d$-evolution family $(\varphi_{s,t})$. Then a family of univalent mappings $(f_t\colon M\to N)$ is an $L^d$-Loewner chain associated with $(\v_{s,t})$ if and only if   it is locally absolutely continuous
on $\R^+$ locally uniformly with respect to $z\in M$ and solves the  Loewner-Kufarev PDE
\begin{equation*}
\frac{\partial f_s}{\partial s}(z)=-(df_s)_zG(z,s),\quad\mbox{a.e. }s\geq 0,z\in M.
\end{equation*}
\end{theorem}

\subsection{The Loewner range} As we saw before, given a $L^d$-evolution family (or just an algebraic evolution family) on a complex manifold, it is well defined the Loewner range ${\sf Lr}(\v_{s,t})$ as the class of biholomorphism of the range of any associated Loewner chain.

Note that if a manifold $N$ is in the Loewner range of an evolution family $(\v_{s,t})$ on a complex manifold $M$, {\sl i.e.}, the biholomorphic class of $N$ coincides with ${\sf Lr}(\v_{s,t})$, then there exists a Loewner chain $(f_t)$ associated to $(\v_{s,t})$ such that $\cup_{t\geq 0}f_t(M)=N$.

One can ask if it is possible to know the Loewner range in case $M$ is a given manifold. For instance, if $M$ is the unit disc $\D\subset \C$, the results of \cite{Co-Di-Gu} imply that the Loewner range of any evolution family of the unit disc is a domain in $\C$ (possibly $\C$ itself as in the classical case). We first start with the following simple remark:

\begin{remark}
If $M$ is simply connected (and non compact) then the Loewner range of any algebraic evolution family of $M$ is simply connected (and non compact). Indeed, if $f_t:M\to N$ is an associated algebraic  Loewner chain, then the Loewner range is biholomorphic to the union of $f_t(M)$ which is an increasing sequence of simply connected domains.
\end{remark}

In particular, if $M=\D$ the unit disc, then the Loewner range of any evolution family on $\D$ is a simply connected non compact Riemann surface, thus, by the uniformization theorem, the Loewner range is either the unit disc $\D$ or $\C$.

In higher dimension the situation is however different: there exists an algebraic evolution family $(\v_{s,t})$ on $\B^3$ which does not admit any associated algebraic Loewner chain with range in $\C^3$ (see \cite[Section 9.4]{A10}). The example is however not regular, and in fact it is not known whether there exists a $L^d$-evolution family on $\B^n$ whose Loewner range does not contain any open domain of $\C^n$.

\medskip

One can somehow try to understand the biholomorphic type of the Loewner range of an evolution family $(\v_{s,t})$ by looking at the dynamics of the family itself. Philosophically this makes sense if one consider the equation $f_t \circ \v_{s,t}=f_s$ as a sort of ``bi-parametric linearization''. The idea is the following: let $\v: M \to M$ be a univalent map. If there exists a univalent map $\sigma: M \to N$, called ``intertwining map'', such that $\sigma\circ \v=\Phi\circ \sigma$, where $\Phi: N \to N$ is an automorphism, one says that $\sigma$ {\sl linearizes} the map $\v$. The automorphism $\Phi$ is generally very simple, but the image $\sigma(M)$ in $N$ might have a complicated geometry, which reflects the dynamics of $\v$.

Starting from this considerations, it is natural to give some answers based on the
asymptotic behavior of the Kobayashi pseudometric under the corresponding evolution family.

\begin{definition}
Let $M$ be a complex manifold. The {\sl Kobayashi pseudometric} $\kappa_M: TM \to \R^+$ is defined by
\[
\kappa_M(z; v):=\inf\{ r>0: \exists g: \D \to M\  \hbox{holomorphic}\ : g(0)=z, g'(0)=\frac{1}{r}v\}.
\]
\end{definition}

The Kobayashi pseudometric has the remarkable property of being contracted by holomorphic maps, and its integrated distance is exactly the Kobayashi pseudodistance. We refer the reader to \cite{Abate} and \cite{Kob} for details.

\begin{definition}
Let $(\v_{s,t})$ be an algebraic evolution family on a complex manifold $M$. For $v\in T_zM$ and $s\geq 0$  we define
\begin{equation}\label{beta}
\beta^s_z(v):=\lim_{t\to \infty} \kappa_M(\v_{s,t}(z); (d\v_{s,t})_z(v)).
\end{equation}
\end{definition}

Since the Kobayashi pseudometric is
contracted by holomorphic mappings the limit in \eqref{beta} is well defined.

The function $\beta$ is the bridge between the dynamics of an algebraic evolution family  $(\v_{s,t})$ and the geometry of its Loewner range. Indeed, in \cite{ABHK} it is proved that, if $N$ is a representative of the Loewner range of $(\v_{s,t})$ and $(f_t:M\to N)$ is an associated algebraic Loewner chain, then
for all $z\in M$ and $v\in T_zM$ it follows
\[
f_s^*\kappa_N(z;v)=\beta^s_z(v).
\]

\begin{remark}
Let $(\v_{s,t})$ be an algebraic evolution family in the unit disc $\D$. The previous formula allows to determine the Loewner range: if $\beta^s_z(v)=0$ for some $s>0, z\in \D$ ($v$ can be taken to be $1$), then the Loewner range is $\C$, otherwise it is $\D$.
\end{remark}

Such a result can be generalized to a complex manifold $M$.
Let ${\sf aut}(M)$ denote the group of
holomorphic automorphisms of a complex manifold $M$. Using a result by J. E. Forn\ae ss and N. Sibony \cite{F-S}, in \cite{ABHK} it is shown that the previous formula implies

\begin{theorem}
Let $M$ be a complete hyperbolic complex manifold and assume that
$M/\sf{aut}(M)$ is compact. Let $(\v_{s,t})$ be an algebraic evolution
family on $M$.  Then
\begin{enumerate}
\item If there exists $z\in M$, $s\geq0$ such that $\beta^s_z(v)\neq 0$ for
all $v\in T_zM$ with $v\neq 0$ then  ${\sf Lr}(\v_{s,t})$ is biholomorphic to $M$.
\item If there exists $z\in M$, $s\geq0$ such that $\dim_\C\{v\in T_zM: \beta_z^s(v)= 0\}=1$  then  ${\sf Lr}(\v_{s,t})$ is a fiber bundle with fiber $\mathbb{C}$
over a closed complex submanifold of $M$.
\end{enumerate}
\end{theorem}

In particular one can apply the previous result to $M=\B^n$ (or even to
the polydiscs in $\C^n$) and obtaining that   any algebraic evolution family    on the unit ball $\B^n$ such that for some $z\in \B^n$, $s\geq 0$ it follows that $\dim_\C\{v\in \C^n : \beta^s_z(v)=0\}\leq 1$, has an open domain in $\C^n$ contained in its Loewner range.

Finally, we note that the Loewner range of an evolution family is strictly related to the so called ``abstract basin of attraction'' of a family of random maps, as studied in hyperbolic dynamics. We refer the reader to \cite{A} for more about this.

\end{document}